\documentclass[12pt,bezier]{article}
\usepackage{times}
\usepackage{booktabs}
\usepackage{pifont}
\usepackage{floatrow}
\usepackage{caption}
\usepackage{mathrsfs}
\usepackage[fleqn]{amsmath}
\usepackage{amsfonts,amsthm,amssymb,mathrsfs,bbding}
\usepackage{txfonts}
\usepackage{graphics,multicol}
\usepackage{graphicx}
\usepackage{color}
\usepackage{mathtools}
\usepackage{amssymb}

\usepackage{bm}
\usepackage{caption}
\usepackage{cite}
\usepackage{latexsym,bm}
\usepackage{indentfirst}
\usepackage{color}
\usepackage[colorlinks=true,anchorcolor=blue,filecolor=blue,linkcolor=blue,urlcolor=blue,citecolor=blue]{hyperref}
\usepackage{extarrows}
\usepackage{cite}
\usepackage{latexsym,bm}

\newtheorem{theorem}{Theorem}[section]
\newtheorem{lemma}[theorem]{Lemma}
\newtheorem{problem}[theorem]{Problem}

\theoremstyle{definition}

\addtocounter{section}{0}

\begin{document}
\title{Maxima of the $Q$-index of non-bipartite $C_{3}$-free graphs}
\author{{\bf Ruifang Liu$^{a}$}, ~{\bf Lu Miao$^{a}$}\thanks{Corresponding author.
E-mail addresses: rfliu@zzu.edu.cn (R. Liu); miaolu0208@163.com (M. Lu);
jie\_xue@126.com (J. Xue).}, ~{\bf Jie Xue$^{a}$}~\\
{\footnotesize $^a$ School of Mathematics and Statistics, Zhengzhou University, Zhengzhou, Henan 450001, China}}

\date{}
\maketitle
{\flushleft\large\bf Abstract}
A classic result in extremal graph theory, known as Mantel's theorem, states that every non-bipartite graph of order $n$ with size $m>\lfloor \frac{n^{2}}{4}\rfloor$  contains a triangle. Lin, Ning and Wu [Comb. Probab. Comput. 30 (2021) 258-270] proved a spectral version of Mantel's theorem for given order $n.$ Zhai and Shu
[Discrete Math. 345 (2022) 112630] investigated a spectral version for fixed size $m.$ In this paper, we prove $Q$-spectral versions of Mantel's theorem.

\begin{flushleft}
\end{flushleft}

\section{Introduction}
Let $A(G)$ be the adjacency matrix and $D(G)$ be the diagonal degree matrix of a graph $G$.
The matrix $Q(G)=D(G)+A(G)$ is known as the $Q$-matrix or the signless Laplacian matrix of $G.$
The largest eigenvalue of $Q(G),$ denoted by $q(G),$ is called the \emph{$Q$-index} or the signless
Laplacian spectral radius of $G.$ By Perron-Frobenius theorem, every connected graph $G$
exists a positive unit eigenvector corresponding to $q(G)$,
which is called the \emph{Perron vector} of $Q(G)$. Note that isolated vertices do not have an effect on the $Q$-index.
Throughout this paper we consider graphs without isolated vertices. Let $n$ and $m$ be the order and the size of $G$, respectively.

A graph is said to be \emph{$H$-free}, if it does not contain $H$ as a subgraph.
Tur\'{a}n-type problem is posed as follows (see \cite{FS}):
For a family $\mathbb{U}$ of graphs which contain no
$H$ as a subgraph, we have two parameters for any $G\in \mathbb{U}$. The goal is to maximize the second parameter under the condition that $G$ is $H$-free
and its first parameter is given.
The investigation of Tur\'{a}n-type problem can date back at least to Mantel in 1907.

\begin{theorem}[\!\cite{Mantel}]
Every graph of order $n$ with $m>\lfloor \frac{n^{2}}{4}\rfloor$ contains a triangle.
\end{theorem}

Note that $m\leq \lfloor \frac{n^{2}}{4}\rfloor$ for every bipartite graph of order $n.$ In fact, Mantel's theorem indicates that every non-bipartite graph of order $n$
with $m>\lfloor \frac{n^{2}}{4}\rfloor$ contains a triangle. Erd\H{o}s \cite{Bondy2008} improved Mantel's theorem and showed that every non-bipartite triangle-free graph of order $n$ and size $m$ satisfies $m\leq\frac{(n-1)^{2}}{4}+1.$

Nikiforov \cite{V5} formally posed a spectral version of Tur\'{a}n-type problem as follows.

\begin{problem}[\!\cite{V5}]\label{p1}
What is the maximum spectral radius of an $H$-free graph of order $n$ or size $m$?
\end{problem}

In the past decades, much attention has been paid to Problem \ref{p1} for some specific graphs $H$ with cycles, see for example, $C_4$ \cite{ZW}, $C_6$ \cite{ZL},
$C_{2k+2}$ \cite{CDM}, $K_{2, r+1}$ \cite{ZLS}, edge-disjoint cycles \cite{LZZ}, friendship graphs \cite{CFTZ}, odd wheels \cite{CS}, intersecting odd cycles \cite{LP}, intersecting cliques \cite{DKL}.

Let $SK_{\lfloor\frac{n-1}{2}\rfloor,\lceil\frac{n-1}{2}\rceil})$ be the graph obtained from $K_{\lfloor\frac{n-1}{2}\rfloor,\lceil\frac{n-1}{2}\rceil}$
by subdividing an edge. Focusing on Problem \ref{p1}, Lin, Ning and Wu \cite{Lin2021} proved a spectral version of Mantel's theorem as follows.

\begin{theorem}[\!\cite{Lin2021}]\label{th1}
Let $G$ be a non-bipartite graph of order $n$. If
\begin{eqnarray*}
\rho(G)\geq\rho(SK_{\lfloor\frac{n-1}{2}\rfloor,\lceil\frac{n-1}{2}\rceil}),
\end{eqnarray*}
then $G$ contains a triangle unless $G\cong SK_{\lfloor\frac{n-1}{2}\rfloor,\lceil\frac{n-1}{2}\rceil}.$
\end{theorem}

\begin{theorem}[\!\cite{Lin2021}]\label{th2}
Let $G$ be a non-bipartite graph of size $m$. If $\rho(G)\geq \sqrt{m-1}$, then $G$ contains a triangle unless $G\cong C_{5}.$
\end{theorem}

Recently, Zhai and Shu \cite{Zhai2022} improved the above result and obtained a sharp spectral condition.

\begin{theorem}[\!\!\cite{Zhai2022}]\label{th3}
Let $G$ be a non-bipartite graph of size $m$. If $\rho(G)\geq\rho(SK_{2,\frac{m-1}{2}})$, then $G$ contains a triangle unless $G\cong SK_{2,\frac{m-1}{2}}$.
\end{theorem}

\input{f1.TpX}\label{1}

Correspondingly, Freitas, Nikiforov and Patuzzi \cite{DF2013} proposed a $Q$-spectral version of Tur\'{a}n-type problem.

\begin{problem}\label{p2}
What is the maximum $Q$-index among all the $H$-free graphs with given order $n$ or size $m$?
\end{problem}

Problem \ref{p2} has been investigated for some given graphs $H$ related to cycles, see for example, $K_{s,t}$ \cite{DF2016},
even cycles \cite{Nikiforov2015}, odd cycles \cite{Yuan2014},
friendship graphs \cite{Zhao2021}, intersecting odd cycles\cite{Chen20}.

Inspired by Theorems \ref{th1}-\ref{th3} and the above results of $Q$-index, we determine the maximum $Q$-index of non-bipartite triangle-free graphs of order $n$ or size $m$. It is observed that the extremal graphs for $Q$-index in Theorems \ref{main1} and \ref{main2} are very different from those for the spectral radius in Theorems \ref{th1}-\ref{th3}.

Let $H$ be a graph with vertex set $V(H)=\{v_{1},v_{2},\ldots,v_{k}\}$ and $r=(n_{1},n_{2},\ldots,n_{k})$ be a vector of positive integers.
A blow-up of $H$, denoted by $H\circ r$, is the graph obtained from $H$ by replacing each vertex $v_{i}$
with an independent set $V_{i}$ of $n_{i}$ vertices and joining each vertex in $V_{i}$ with each vertex in $V_{j}$ for $v_{i}v_{j}\in E(H)$.
The graph $C_{5}\circ(n-4,1,1,1,1)$ (see Fig. \ref{1}) is a blow-up of $C_{5}$, where $V(C_{5})=\{v_{1},v_{2},v_{3},v_{4},v_{5}\}$ and $r=(n-4,1,1,1,1)$.
Moreover, $q(C_{5}\circ(n-4,1,1,1,1))$ is the largest root of $$x^{3}-(n+2)x^{2}+(3n-2)x-4=0.$$ Next we present the first main result.

\begin{theorem}\label{main1}
Let $G$ be a non-bipartite graph of order $n$. If
$$q(G)\geq q(C_{5}\circ(n-4,1,1,1,1)),$$
then $G$ contains a triangle unless $G\cong C_{5}\circ(n-4,1,1,1,1)$.
\end{theorem}

If $m>\lfloor\frac{n^{2}}{4}\rfloor$, then $m>\frac{n^{2}-1}{4}$. By a well-known inequality $q(G)\geq\frac{4m}{n}$, we have
$q(G)\geq\frac{4m}{n}>\frac{n^{2}-1}{n}=n-\frac{1}{n}>q(C_{5}\circ(n-4,1,1,1,1))$. This implies that Theorem \ref{main1} is stronger than Mantel's theorem.

Denote by $C_{5}\bullet K_{1,m-5}$ the graph by identifying a vertex of $C_{5}$ and the central vertex of $K_{1,m-5}$ (see Fig. \ref{2}).
Furthermore, $q(C_{5}\bullet K_{1,m-5})$ is the largest root of $$x^{4}-(m+3)x^{3}+(5m-5)x^{2}+(8-5m)x+4=0.$$

\begin{theorem}\label{main2}
Let $G$ be a non-bipartite graph of size $m$. If
\begin{eqnarray*}
q(G)\geq q(C_{5}\bullet K_{1,m-5}),
\end{eqnarray*}
then $G$ contains a triangle unless $G\cong C_{5}\bullet K_{1,m-5}.$
\end{theorem}

\input{f2.TpX}

In the following sections, we first prepare the grounds for the proofs of Theorems \ref{main1} and \ref{main2}, and then give their proofs.

\section{Proof of Theorem \ref{main1}}

We in this section will present some important results that will be used in our arguments.

Let $M$ be a real $n\times n$ matrix, and let $V(G)=\{1,2,\ldots,n\}$. Given a partition $\Pi: V(G)=V_{1}\cup V_{2}\cup\cdots\cup V_{k}$, $M$ can be correspondingly partitioned as
$$
M=
\begin{pmatrix}
M_{1,1}&M_{1,2}&\cdots&M_{1,k}\\
M_{2,1}&M_{2,2}&\cdots&M_{2,k}\\
\vdots&\vdots&\ddots&\vdots\\
M_{k,1}&M_{k,2}&\cdots&M_{k,k}
\end{pmatrix}
.$$
The quotient matrix of $M$ with respect to $\Pi$ is defined as the $k\times k$ matrix $B_{\Pi} = (b_{i,j})^{k}_{i,j=1}$, where $b_{i,j}$ is the average value of all row sums of $M_{i,j}$. The partition $\Pi$ is called equitable if each block $M_{i,j}$ of $M$ has constant row sum $b_{i,j}$. Also, we say that the quotient matrix $B_{\Pi}$ is equitable if $\Pi$ is an equitable partition of $M$.

\begin{lemma}[\!\cite{Brouwer2011,Godsil2001}]\label{lem2}
 Let $M$ be a real symmetric matrix and $\lambda(M)$ be its largest eigenvalue.
 Let $B_{\Pi}$ be an equitable quotient matrix of $M$. Then the eigenvalues of $B_{\Pi}$ are also eigenvalues of $M$.
 Furthermore, if $M$ is nonnegative and irreducible, then $\lambda(M)=\lambda(B_{\Pi})$.
\end{lemma}

\begin{lemma}[\!\cite{Hong2005}]\label{lem3}
Let $G$ be a connected graph and $X=(x_{v})_{v\in V (G)}$ be the Perron vector of $Q(G)$.
Assume that $u_{1}v\notin E(G)$ while $u_{2}v\in E(G)$. If $x_{u_{1}}\geq x_{u_{2}}$, then $q(G-u_{2}v+u_{1}v)>q(G)$.
\end{lemma}

The following two lemmas give upper bounds on the $Q$-index.

\begin{lemma}[\!\cite{CRS2007}]\label{lem5}
For every graph $G$,
\begin{eqnarray*}
q(G)\leq \max\{d(u)+d(v): uv\in E(G)\}.
\end{eqnarray*}
For a connected $G$, equality holds if and only if $G$ is regular graph or semi-regular bipartite.
\end{lemma}

\begin{lemma}[\!\cite{Merris1998, Feng2009}]\label{lem4}
For every graph $G$,
\begin{eqnarray*}
q(G)\leq \max\{d(u)+m(u): u\in V(G)\},
\end{eqnarray*}
where $m(u)=\frac{1}{d(u)}\sum\limits_{v\in N(u)}d(v)$.
If $G$ is connected, then the equality holds if and only if $G$ is either a regular graph or a semi-regular bipartite graph.
\end{lemma}

The $Q$-index of $C_{5}\circ(n-4,1,1,1,1)$ is estimated in the following lemma.

\begin{lemma}\label{lem6}
Let $q(C_{5}\circ(n-4,1,1,1,1))$ be the $Q$-index of the graph $C_{5}\circ(n-4,1,1,1,1)$. Then the following results hold.\\
(i) $q(C_{5}\circ(n-4,1,1,1,1))$ is the largest root of $x^{3}-(n+2)x^{2}+(3n-2)x-4=0$.\\
(ii) $n-1-\frac{1}{n}<q(C_{5}\circ(n-4,1,1,1,1))<n-\frac{1}{n}$ for $n\geq5$.
\end{lemma}

\begin{proof}
(i) We partition the vertex set of $C_{5}\circ(n-4,1,1,1,1)$ as $\Pi$: $V_{1}\cup V_{2}\cup V_{3}$,
where $V_{2}=\{v_{2},v_{5}\}$, $V_{3}=\{v_{3},v_{4}\}$ and $V_{1}$ contains the remaining vertices (see Fig. \ref{1}).
The quotient matrix $B_{\Pi}$ of the $Q$-matrix $Q(C_{5}\circ(n-4,1,1,1,1))$ with respect to the partition $\Pi$ is as follows:
$$
B_{\Pi}=
\begin{pmatrix}
2&2&0\\
n-4&n-3&1\\
0&1&3
\end{pmatrix}.
$$
Then the characteristic polynomial of $B_{\Pi}$ is
\begin{eqnarray*}
f(x)=|xI_{3}-B_{\Pi}|=x^{3}-(n+2)x^{2}+(3n-2)x-4.
\end{eqnarray*}
Note also that $\Pi$ is an equitable partition.
According to Lemma \ref{lem2}, it follows that $q(C_{5}\circ(n-4,1,1,1,1))$ is the largest root of $x^{3}-(n+2)x^{2}+(3n-2)x-4=0$.
Hence the result $(i)$ follows.

(ii) Since $f(n-1-\frac{1}{n})<0$, then $q(C_{5}\circ(n-4,1,1,1,1))>n-1-\frac{1}{n}$.
Note that $f'(x)=3x^{2}-2(n+2)x+3n-2>0$ for $x\geq n-\frac{1}{n}$ and $n\geq5$.
It follows that $f(x)\geq f(n-\frac{1}{n})>0.$
Hence $q(C_{5}\circ(n-4,1,1,1,1))<n-\frac{1}{n}$, and $(ii)$ holds.
\end{proof}

Assume that $G$ is the graph with the maximum $Q$-index among all the non-bipartite triangle-free graphs of order $n$. We first prove a
general structure of the extremal graph $G$.

\begin{lemma}\label{lem10}
Let $G$ attain the maximum $Q$-index among all the non-bipartite triangle-free graphs of order $n$. Then we have
$$ G\cong C_{5}\circ(n_{1},n_{2},n_{3},n_{4}, 1).$$
\end{lemma}

\begin{proof}
Suppose that $G$ attains the maximum $Q$-index among all the non-bipartite triangle-free graphs of order $n$. Then we have $n\geq5.$
Our goal is to show that $G\cong C_{5}\circ(n-4,1,1,1,1)$.
First we claim that $G$ is connected.
If not, we assume that $G_{1}$ is a connected component of $G$ with $q(G_{1})=q(G)$.
By adding an edge between $G_{1}$ and any other connected component of $G,$ we get a new graph $G_{2}.$
Note that $G_{2}$ is still a non-bipartite triangle-free graph. However, $q(G_{2})>q(G_{1})=q(G),$ which contradicts the maximality of $q(G).$

Let $u'v'$ be the edge of $G$ with $d(u')+d(v')=\max\{d(u)+d(v): uv\in E(G)\}$.
According to Lemmas \ref{lem5} and \ref{lem6}, we have
\begin{eqnarray}\label{eq1}
n-2<n-1-\frac{1}{n}<q(C_{5}\circ(n-4,1,1,1,1))\leq q(G)\leq d(u')+d(v').
\end{eqnarray}
By (\ref{eq1}), it follows that $d(u')+d(v')\geq n-1$. Note that $G$ is triangle-free. Then $u'$ and $v'$ have no common neighbours, and this implies that $d(u')+d(v')\leq n.$
If $d(u')+d(v')=n,$ then $G$ is a bipartite graph, a contradiction, and hence $d(u')+d(v')\leq n-1.$
From the above two aspects, we always have $d(u')+d(v')=n-1$, and this implies that there exists a vertex $w$ of $G$ such that $w\notin N(u')\cup N(v')$ (see Fig. \ref{3}).
\input{f3.TpX}\label{f3}

Define $N_{1}=\{v\in N(u') : vw\in E(G)\}$, $N_{2}=\{v\in N(v'): vw\in E(G)\}$, $N_{3}=N(u')\backslash (N_{1}\cup \{v'\})$ and $N_{4}=N(v')\backslash (N_{2}\cup \{u'\}).$ Since $G$ is non-bipartite, one can see that $N_{1}\neq \emptyset$ and $N_{2}\neq \emptyset$.
Note that both $N_{1}$ and $N_{4}$ are independent sets as $G$ is triangle-free.
We claim that $G[N_{1}\cup N_{4}]$ is a complete bipartite graph.
Suppose that there exist two vertices $v_{1}\in N_{1}$ and $v_{2}\in N_{4}$ such that $v_{1}v_{2}\notin E(G).$
Then $G+v_{1}v_{2}$ is still non-bipartite triangle-free.
However, $q(G+v_{1}v_{2})>q(G),$ a contradiction.
By the same analysis, $G[N_{2}\cup N_{3}]$ and $G[N_{3}\cup N_{4}]$ are also complete bipartite graphs (see Fig. \ref{3}).

Let $V_{1}=N_{1}$, $V_{2}=\{u'\}\cup N_{4}$, $V_{3}=N_{3}\cup\{v'\}$ and $V_{4}=N_{2}.$ Then $\cup_{i=1}^{4}V_{i}=V(G)\backslash\{w\}.$
Set $|V_{i}|=n_{i}$ for $1\leq i\leq 4.$ Then we have $\sum_{i=1}^{4}n_{i}=n-1.$
Recall that $N_{1}\neq \emptyset, u'\in V_{2}, v'\in V_{3}$ and $N_{2}\neq \emptyset.$ Hence we have $n_{i}\geq1,$ where $1\leq i\leq 4.$
It is easy to see that $G\cong C_{5}\circ(n_{1},n_{2},n_{3},n_{4}, 1)$ (see Fig. \ref{3}).
\end{proof}

Now we are ready to give the proof of Theorem \ref{main1}.

\medskip
\noindent\textbf{Proof of Theorem \ref{main1}.}
Suppose that $G$ attains the maximum $Q$-index among all the non-bipartite triangle-free graphs of order $n$. By Lemma \ref{lem10}, we have
$G\cong C_{5}\circ(n_{1},n_{2},n_{3},n_{4}, 1).$ Let $X=(x_{v})_{v\in V (G)}$ be the Perron vector of $Q(G)$.
Let $u^{\ast}$ be a vertex of $G$ such that $d(u^{\ast})+m(u^{\ast})=\max\{d(u)+m(u):u\in V(G)\}.$
By Lemma \ref{lem4}, we have $q(G)\leq d(u^{\ast})+m(u^{\ast})$. Next we discuss the following three cases according to different $u^{\ast}.$

\vspace{1.5mm}
\noindent{\bf Case 1.} $u^{\ast}=w$.
\vspace{0.5mm}

Without loss of generality, we may assume that $n_{2}\leq n_{3}$. It follows that
\begin{eqnarray}\label{eq2}
d(w)+m(w)=n_{1}+n_{4}+\frac{n_{1}(1+n_{2})+n_{4}(1+n_{3})}{n_{1}+n_{4}}
\leq n_{1}+n_{4}+n_{3}+1=n-n_{2}.
\end{eqnarray}
By Lemma \ref{lem6}, we know that $q(G)>n-1-\frac{1}{n}.$ Combining Lemma \ref{lem4} and (\ref{eq2}), we have $$n-1-\frac{1}{n}<q(G)\leq d(w)+m(w)\leq n-n_{2}.$$
Hence $n_{2}<1+\frac{1}{n}$, and thus we have $n_{2}=1.$
By taking $n_{2}=1$ in (\ref{eq2}), we have
\begin{eqnarray*}
n-1-\frac{1}{n}<d(w)+m(w)=n-1-n_{3}+\frac{n_{1}+n_{3}n_{4}}{n_{1}+n_{4}}=n-1-\frac{n_{1}(n_{3}-1)}{n_{1}+n_{4}},
\end{eqnarray*}
which leads to $\frac{1}{n}>\frac{n_{1}(n_{3}-1)}{n_{1}+n_{4}}$, and hence $n_{3}=1$. Recall that $n_{4}\geq1.$
So we can choose a vertex $v_{4}\in V_{4}.$
If $n_{4}\geq 2,$ then there exists at least one vertex $z\in V_{4}\backslash \{v_{4}\}.$
Note that $V_{2}=\{u'\}$ and $V_{3}=\{v'\}$. Without loss of generality, we assume that $x_{u'}\geq x_{v'}$.
Define a new graph $G'=G-zv'+zu'$ for $z\in V_{4}\backslash \{v_{4}\}.$
Note that $G'$ is still a non-bipartite triangle-free graph.
According to Lemma \ref{lem3}, we have $q(G')>q(G)$, a contradiction. Hence $n_{4}=1,$ and then $n_{1}=n-4.$ So we have $G\cong C_{5}\circ(n-4,1,1,1,1).$

\vspace{1.5mm}
\noindent{\bf Case 2.} $u^{\ast}\in V_{1}\cup V_{4}.$ \label{case2}
\vspace{0.5mm}

By symmetry, we may assume that $u^{\ast}\in V_{1}.$ Then we have
\begin{eqnarray}\label{eq3}
d(u^{\ast})+m(u^{\ast})&=&n_{2}+1+\frac{n_{1}+n_{4}+n_{2}(n_{1}+n_{3})}{n_{2}+1}=n_{2}+1+n_{1}+\frac{n_{4}+n_{2}n_{3}}{n_{2}+1}\nonumber\\
&=&n_{2}+1+n_{1}+n_{3}+\frac{n_{4}-n_{3}}{n_{2}+1}=n-n_{4}+\frac{n_{4}-n_{3}}{n_{2}+1}.
\end{eqnarray}
By Lemma \ref{lem6} and (\ref{eq3}), we have $$n-1-\frac{1}{n}<q(G)\leq d(u^{\ast})+m(u^{\ast})=n-n_{4}+\frac{n_{4}-n_{3}}{n_{2}+1}.$$
It follows that
\begin{eqnarray*}
n_{2}(n_{4}-1-\frac{1}{n})+n_{3}-1-\frac{1}{n}<0,
\end{eqnarray*}
and hence we have $n_{3}=1$ and $n_{4}=1.$ This implies that $G\cong C_{5}\circ(n_{1},n_{2},1,1,1).$
The quotient matrix $B_{\Pi}$ of $Q(G)$ on the partition $\Pi: V(G)=V_{1}\cup V_{2}\cup V_{3} \cup V_{4}\cup\{w\}$ is as follows:
$$B_{\Pi}=
\begin{pmatrix}
1+n_{2}&n_{2}&0&0&1\\
n_{1}&1+n_{1}&1&0&0\\
0&n_{2}&1+n_{2}&1&0\\
0&0&1&2&1\\
n_{1}&0&0&1&1+n_{1}
\end{pmatrix}
.$$
By a direct calculation, the characteristic polynomial $f(n_{1},n_{2},x)$ of the quotient matrix $B_{\Pi}$ is
\begin{eqnarray*}
f(n_{1},n_{2},x)&=&x^{5}-2(n_{1}+n_{2}+3)x^{4}+\Big((n_{1}+n_{2})^{2}+9(n_{1}+n_{2})+n_{1}n_{2}+12\Big)x^{3}\\
&&-\Big(3(n_{1}+n_{2})^{2}+11(n_{1}+n_{2})+n_{1}n_{2}(n_{1}+n_{2}+4)+10\Big)x^{2}\\
&&+\Big((n_{1}+n_{2})^{2}+4(n_{1}+n_{2})+2n_{1}n_{2}(n_{1}+n_{2}+3)+3\Big)x\\
&&-4n_{1}n_{2}.
\end{eqnarray*}
Since the partition $\Pi$ is equitable, it follows from Lemma \ref{lem2} that
$q(G)$ is equal to the largest root of $f(n_{1},n_{2},x)=0.$ Without loss of generality, we assume that $n_{1}\geq n_{2}.$
Next we will prove that $n_{2}=1$. Suppose to the contrary that $n_{2}\geq 2$.
Define $G'=C_{5}\circ(n_{1}+1,n_{2}-1,1,1,1)$. Obviously, $G'$ is still non-bipartite triangle-free.
By Lemma \ref{lem2}, $q(G')$ equals the largest root of $f(n_{1}+1,n_{2}-1,x)=0$. Note that
\begin{eqnarray*}
f(n_{1}+1,n_{2}-1,x)&=&x^{5}-2(n_{1}+n_{2}+3)x^{4}+\Big((n_{1}+n_{2})^{2}+9(n_{1}+n_{2})\\
&&+(n_{1}+1)(n_{2}-1)+12\Big)x^{3}-\Big(3(n_{1}+n_{2})^{2}+11(n_{1}+n_{2})\\
&&+(n_{1}+1)(n_{2}-1)(n_{1}+n_{2}+4)+10\Big)x^{2}+\Big((n_{1}+n_{2})^{2}\\
&&+4(n_{1}+n_{2})+2(n_{1}+1)(n_{2}-1)(n_{1}+n_{2}+3)+3\Big)x\\
&&-4(n_{1}+1)(n_{2}-1).
\end{eqnarray*}
Hence we have
\begin{eqnarray*}
f(n_{1},n_{2},x)-f(n_{1}+1,n_{2}-1,x)=(2n_{1}-n+4)(x-2)(x^{2}-(n-1)x+2).
\end{eqnarray*}
Since $q(G)>n-1-\frac{1}{n}\geq\frac{19}{5}$, then we have $x-2>0$ and $x^{2}-(n-1)x+2>0$ for any $x\geq q(G).$
Note also that $2n_{1}-n+4>0.$ Hence we have $f(n_{1},n_{2},x)>f(n_{1}+1,n_{2}-1,x)$ for any $x\geq q(G).$
So we obtain that $q(G')>q(G)$, a contradiction. Therefore, $n_{2}=1.$ This implies that $G\cong C_{5}\circ(n-4,1,1,1,1),$ as desired.

\vspace{1.5mm}
\noindent{\bf Case 3.} $u^{\ast}\in V_{2}\cup V_{3}.$
\vspace{0.5mm}

Without loss of generality, we assume that $u^{\ast}\in V_{2}$. Then we have
\begin{eqnarray}\label{eq4}
d(u^{\ast})+m(u^{\ast})&=&n_{1}+n_{3}+\frac{n_{1}(1+n_{2})+n_{3}(n_{2}+n_{4})}{n_{1}+n_{3}}=n_{1}+n_{2}+n_{3}+\frac{n_{1}+n_{3}n_{4}}{n_{1}+n_{3}}\nonumber\\
&=&n_{1}+n_{2}+n_{3}+n_{4}-\frac{n_{1}(n_{4}-1)}{n_{1}+n_{3}}=n-1-\frac{n_{1}(n_{4}-1)}{n_{1}+n_{3}}.
\end{eqnarray}
Since $d(u^{\ast})+m(u^{\ast})>n-1-\frac{1}{n}$, it follows from (\ref{eq4}) that $\frac{1}{n}>\frac{n_{1}(n_{4}-1)}{n_{1}+n_{3}}.$
This implies that $n_{4}=1$. Let $V_{4}=\{v_{4}\}$. Without loss of generality, we assume that $x_{w}\geq x_{v_{4}}.$
Note that $n_{3}\geq1.$ Then we can choose a vertex $v_{3}\in V_{3}.$ If $n_{3}\geq 2$, then there exists a vertex
$z\in V_{3}\backslash \{v_{3}\}.$ Let $G'=G-zv_{4}+zw.$ Obviously, $G'$ is non-bipartite triangle-free.
However, by Lemma \ref{lem3}, we have $q(G')>q(G),$ which contradicts the maximality of $q(G)$. Hence $n_{3}=1$, that is, $G\cong C_{5}\circ(n_1,n_{2},1,1,1)$.
Next similar to the argument in Case \ref{case2}, we can prove that $G\cong C_{5}\circ(n-4,1,1,1,1).$ This completes the proof.
\hspace*{\fill}$\Box$

\section{Proof of Theorem \ref{main2}}

In the section, we first present the $Q$-index of the extremal graph $C_{5}\bullet K_{1,m-5}.$

\begin{lemma}
Let $q(C_{5}\bullet K_{1,m-5})$ be the $Q$-index of $C_{5}\bullet K_{1,m-5}$. Then
$q(C_{5}\bullet K_{1,m-5})$ is the largest root of $x^{4}-(m+3)x^{3}+(5m-5)x^{2}+(8-5m)x+4=0.$
\end{lemma}

\begin{proof}
We partition the vertex set of $C_{5}\bullet K_{1,m-5}$ as $\Pi$: $V_{1}\cup V_{2}\cup V_{3}\cup V_{4}$,
where $V_{1}=\{v_{1},v_{2}\}$, $V_{2}=\{v_{3},v_{5}\}$, $V_{4}=\{v_{4}\}$ and $V_{3}$ contains the remaining vertices (see Fig. \ref{2}).
The quotient matrix $B_{\Pi}$ of the $Q$-matrix $Q(C_{5}\bullet K_{1,m-5})$ with respect to the partition $\Pi$ is
$$
B_{\Pi}=
\begin{pmatrix}
3&1&0&0\\
1&2&0&1\\
0&0&1&1\\
0&2&m-5&m-3
\end{pmatrix}.
$$
Then the characteristic polynomial of $B_{\Pi}$ is as follows:
\begin{eqnarray*}
f(x)=|xI_{4}-B_{\Pi}|=x^{4}-(m+3)x^{3}+(5m-5)x^{2}+(8-5m)x+4.
\end{eqnarray*}
Note that $\Pi$ is an equitable partition.
By Lemma \ref{lem2}, it follows that $q(C_{5}\bullet K_{1,m-5})$ is the largest root of $x^{4}-(m+3)x^{3}+(5m-5)x^{2}+(8-5m)x+4=0.$
\end{proof}

Now, we are in a position to present the proof of Theorem \ref{main2}.

\medskip
\noindent  \textbf{Proof of Theorem \ref{main2}}.
Suppose that $G$ attains the maximum $Q$-index among all the non-bipartite triangle-free graphs of size $m.$
Note that a non-bipartite graph $G$ of size $m\leq5$ must contain a triangle unless $G\cong C_{5}$.
In the following, it suffices to consider the case $m\geq6.$
Our aim is to show that $G\cong C_{5}\bullet K_{1,m-5}$.
First we claim that $G$ is connected.
Otherwise, suppose that $G$ contains $k$ connected components $G_{1},G_{2},\ldots,G_{k}$,
where $q(G)=q(G_{i_{0}})$ for some $i_{0}\in \{1,2,\ldots,k\}$.
Select a vertex $u_{i}\in V(G_{i})$ for each $i\in \{1,2,\ldots,k\},$ and let $G'$ be the graph obtained from
$G$ by identifying vertices $u_{1},u_{2},\ldots,u_{k}.$
It is clearly that $G'$ is also a non-bipartite triangle-free graph of size $m$.
Moreover, it contains $G_{i_{0}}$ as a proper subgraph, and hence $q(G')>q(G_{i_{0}})=q(G)$, contradicting the maximality of $q(G).$

Note that $C_{5}\bullet K_{1,1}$ is the only non-bipartite triangle-free graph with size $m=6$. Hence the theorem holds trivially for $m=6.$
Next we consider the case $m=7$. Let $X=(x_{v})_{v\in V (G)}$ be the Perron vector of $Q(G)$. Let $C$ be the shortest odd cycle of $G$.
If $|V(C)|=7$, then $G\cong C_{7}$. However, $q(C_{7})=4<q(C_{5}\bullet K_{1,2})$, a contradiction.
Then $|V(C)|=5,$ and $G$ must be isomorphic to one of the following graphs: $H_{1}$, $H_{2}$, $H_{3}$ (see Fig. \ref{4}) and $C_{5}\bullet K_{1,2}$.
If $G\cong H_{1}$ or $G\cong H_{2}$, then $d(u)+d(v)\leq5$ for each edge $uv\in E(G).$ By Lemma \ref{lem5}, we have $q(G)\leq5<q(C_{5}\bullet K_{1,2})$,
which contradicts the maximality of $q(G).$
If $G\cong H_{3},$ without loss of generality, we assume that $x_{u_{1}}\geq x_{u_{2}}$, where $u_{1}$ and $u_{2}$ are the neighbours of $v_{1}$ and $v_{2}$, respectively. According to Lemma \ref{lem3}, we have $q(G-v_{2}u_{2}+v_{2}u_{1})>q(G)$, a contradiction. Hence $G\cong C_{5}\bullet K_{1,2}$, and the theorem holds for $m=7.$
\input{f4.TpX}

Next we assume that $m\geq 8$. Note that $K_{1,m-3}$ is a proper subgraph of $C_{5}\bullet K_{1,m-5}$
and $q(K_{1,m-3})=m-2$, then $q(C_{5}\bullet K_{1,m-5})>m-2$.
Combining the fact $q(G)\geq q(C_{5}\bullet K_{1,m-5})$, it follows that
\begin{eqnarray}\label{eq10}
q(G)>m-2.
\end{eqnarray}
Let $\Delta$ be the maximum degree of $G.$ If $\Delta\geq m-2,$ then either $G$ is bipartite or $G$ contains a triangle, a contradiction.
Hence we have $\Delta \leq m-3.$
Let $u^{\ast}$ be a vertex of $G$ such that $d(u^{\ast})+m(u^{\ast})=\max\{d(u)+m(u): u\in V(G)\}$.
By Lemma \ref{lem4} and (\ref{eq10}), we have
\begin{eqnarray}\label{eq7}
  d(u^{\ast})+m(u^{\ast})=d(u^{\ast})+\frac{\sum_{v\in N(u^{\ast})}d(v)}{d(u^{\ast})}\geq q(G)>m-2.
\end{eqnarray}
Now we divide the proof into three cases according to different values of $d(u^{\ast}).$

\vspace{1.5mm}
\noindent{\bf Case 1.} $d(u^{\ast})=1$.
\vspace{0.5mm}

Let $v$ be the only neighbour of $u^{\ast}.$ It follows from (\ref{eq7}) that $1+d(v)>m-2$, and hence $d(v)>m-3,$ which contradicts the fact $\Delta\leq m-3$.

\vspace{1.5mm}
\noindent{\bf Case 2.} $2\leq d(u^{\ast})\leq m-4$.
\vspace{0.5mm}

Let $N(u^{\ast})$ be the set of neighbours of $u^{\ast}$ in $G$. Since $G$ is triangle-free, then $N(u^{\ast})$ is an independent set.
Hence we have
\begin{eqnarray*}
\sum_{v\in N(u^{\ast})}d(v)=|E\big(N(u^{\ast}), V(G)\backslash N(u^{\ast})\big)|\leq m,
\end{eqnarray*}
where $E\big(N(u^{\ast}), V(G)\backslash N(u^{\ast})\big)$ denotes the set of edges between $N(u^{\ast})$ and $V(G)\backslash N(u^{\ast})$.
It follows that
\begin{eqnarray*}
d(u^{\ast})+m(u^{\ast})=d(u^{\ast})+\frac{\sum_{v\in N(u^{\ast})}d(v)}{d(u^{\ast})}\leq d(u^{\ast})+\frac{m}{d(u^{\ast})}.
\end{eqnarray*}
Note that the function $f(x)=x+\frac{m}{x}$ is convex for $x>0.$
Hence the maximum of the expression $d_{u^{\ast}}+\frac{m}{d_{u^{\ast}}}$ is attained at $d_{u^{\ast}}=2$ or $d_{u^{\ast}}=m-4$.
Since $m\geq8$, then we have
\begin{eqnarray*}
d(u^{\ast})+\frac{m}{d(u^{\ast})}\leq\max\{2+\frac{m}{2}, m-4+\frac{m}{m-4}\}\leq m-2,
\end{eqnarray*}
which contradicts (\ref{eq7}).

\vspace{1.5mm}
\noindent{\bf Case 3.} $d(u^{\ast})=m-3$.
\vspace{0.5mm}

Let $T=V(G)\backslash (N(u^{\ast})\cup\{u^{\ast}\})$, and let $e(T)$ denote the number of edges in the induced subgraph $G[T]$.
Since $G$ is triangle-free, then $N(u^{\ast})$ is an independent set.
Note that $G$ is a connected non-bipartite graph. Hence $e(T)=1$. Let $e_{0}$ be the only edge in $G[T]$.
This implies that there are exactly two edges between $N(u^{\ast})$ and $T$.
Note also that $G$ is non-bipartite triangle-free. Then these two edges are independent and both of them are adjacent to $e_{0}$,
which implies that $G\cong C_{5}\bullet K_{1,m-5}$. This completes the proof.\hspace*{\fill}$\Box$
\medskip

\section*{Acknowledgements}

%The authors would like to thank the anonymous referees for their helpful comments on improving the presentation of the paper.
The research of Ruifang Liu is supported by National Natural Science Foundation of China (Nos. 11971445 and 12171440) and Natural Science Foundation of Henan (No. 202300410377). The research of Jie Xue is supported by National Natural Science Foundation of China (No. 12001498)
and China Postdoctoral Science Foundation (No. 2022TQ0303).

\end{document}